\documentclass[12pt]{amsart}

\usepackage[left=1in, right=1in, top=1in, bottom=1in]
{geometry}
\usepackage{amsfonts,amssymb,graphicx,amsmath,amsthm, float}
\usepackage{xypic}
\usepackage[all]{xy}

\theoremstyle{plain}
\newtheorem{theorem}{Theorem}[section]

\newtheorem{lemma}[theorem]{Lemma}
\newtheorem{example}[theorem]{Example}

\newtheorem{corollary}[theorem]{Corollary}

\newcommand{\N}{\mathbb{N}}

\theoremstyle{remark}
\newtheorem*{remark}{Remark}

\newcommand{\Z}{\mathbb{Z}}

\begin{document}

\title{On Irreducible Divisor Graphs in Commutative Rings with Zero-Divisors}
        \date{\today}
\author{Christopher Park Mooney}
\address{Reinhart Center \\ Viterbo University \\ 900 Viterbo Drive \\ La Crosse, WI 54601}
\email{cpmooney@viterbo.edu}



\begin{abstract} In this paper, we continue the program initiated by I. Beck's now classical paper concerning zero-divisor graphs of commutative rings.  After the success of much research regarding zero-divisor graphs, many authors have turned their attention to studying divisor graphs of non-zero elements in the ring, the so called irreducible divisor graph.  In this paper, we construct several different associated irreducible divisor graphs of a commutative ring with unity using various choices for the definition of irreducible and atomic in the literature.  We continue pursuing the program of exploiting the interaction between algebraic structures and associated graphs to further our understanding of both objects.  Factorization in rings with zero-divisors is considerably more complicated than integral domains; however, we find that many of the same techniques can be extended to rings with zero-divisors.  This allows us to not only find graph theoretic characterizations of many of the finite factorization properties that commutative rings may possess, but also understand graph theoretic properties of graphs associated with certain commutative rings satisfying nice factorization properties.
\\
\vspace{.1in}\noindent \textbf{2010 AMS Subject Classification:} 13A05, 13E99, 13F15, 5C25

\end{abstract}
\maketitle
\section{Introduction}  
\indent In this article, $R$ will denote a commutative ring with unity, not equal to zero.   Let $R^*=R-\{0\}$, $U(R)$ be the units of $R$, and $R^{\#}=R^*-U(R)$, the non-zero, non-units of $R$.  We will use $D$ to denote an integral domain.  We will use $G=(V,E)$ to denote a graph $G$ with $V$, the set of vertices, and $E$, the set of edges.  Our graphs will be undirected and not necessarily simple (we allow loops but no multi-edges).  We will denote an edge between vertices $a,b \in V$ by juxtaposition, as in $ab\in E$.
\\
\indent Recently, the study of the relationship between graphs and rings has become quite popular.  In many ways this program began with the now classic paper in 1988, by Istvan Beck, \cite{Beck}.  He introduced, for a commutative ring $R$, the notion of a zero-divisor graph $\Gamma(R)$.  Traditionally, the vertices of $\Gamma(R)$ are the set of zero-divisors and there is an edge between distinct $a,b \in Z(R)$ if $ab=0$.  One thing to note is that this is a simple graph and so there are no loops even if $x^2=0$.  This has been the subject of some debate as to whether one should allow loops or not.  Another modification of the original zero-divisor graph that has become quite standard is to remove $0$ from the vertex set, so $V=Z(R)^*$.  The zero-divisor graph has attracted a significant amount of attention recently having been studied and developed by many authors including, but not limited to D.D. Anderson, D.F. Anderson, M. Axtell, A. Frazier, J. Stickles, A. Lauve, P.S. Livingston, and M. Naseer in \cite{Andersonzdg, Axtellzdg, Davidanderson, Livingston, Mooney}.  
\\
\indent There have been several generalizations and extensions of this concept, but in this paper, we focus on the notion of an irreducible divisor graph first formulated by J. Coykendall and J. Maney in \cite{Coykendall} for integral domains.  Instead of looking exclusively at the divisors of zero in a ring, the authors restrict to a domain $D$ and choose any non-zero, non-unit $x\in D$.  They study the relationships between the irreducible divisors of $x$.  This study provides much insight into many of the factorization properties of the domain by providing a graphical representation of the multiplicative structure.  Recently, M. Axtell, N. Baeth, and J. Stickles presented several nice results about factorization properties of domains based on their associated irreducible divisor graphs, in \cite{Axtellidgd}.  They have also extended these definitions of irreducible divisor graphs to rings with zero-divisors using a particular choice of irreducible and associate, in \cite{Axtellidgzd}.  Generalized $\tau$-factorization techniques have been applied to irreducible divisor graphs in integral domains to study $\tau$-finite factorization properties using $\tau$-irreducible divisor graphs in \cite{Mooney3}
\\
\indent When zero-divisors are present, choosing the definition of irreducible and associate becomes a bit more complicated.  In \cite{Valdezleon}, D.D. Anderson and S. Valdez-Leon study several distinct choices for irreducible and associate that various authors have used over the years when looking at factorization in rings with zero-divisors.  In \cite{Axtellidgzd}, the authors choose to use $a$ and $b$ are associates, written $a\sim b$ if $(a)=(b)$.  They say $a$ is irreducible if $a=bc$ then $(a)=(b)$ or $(a)=(c)$.  They then construct irreducible divisor graphs in a natural way to attain some very nice results.
\\
\indent In this paper, we are interested in extending irreducible divisor graphs to work with the many other notions of irreducible and associate which exist in the literature.  This will enable us to extend many theorems to work with a wider range of finite factorization properties that commutative rings with zero-divisors may possess.  Because the definitions for irreducible and associate chosen previously in the literature are the weakest, we find that we are even able to prove several stronger theorems using more powerful notions of irreducible and associate.
\\
\indent Section Two provides the requisite preliminary background information and factorization definitions from the literature regarding rings with zero-divisors as well as much of the definitions from the study of irreducible and zero-divisor graphs.  In Section Three, we define a variety of irreducible divisor graphs of a commutative ring $R$ and examine the relationship between these different graphs.  In Section Four, we provide an example in which we study the various irreducible divisor graphs associated with a particular irreducible element in the ring $\Z \times \Z$.  We are especially interested in comparing these with the irreducible divisor graphs of \cite{Axtellidgd, Coykendall} of irreducible elements in integral domains. In Section Five, we prove several theorems illustrating how irreducible divisor graphs give us another way to characterize various finite factorization properties rings may possess as defined in \cite{Valdezleon}.

\section{Preliminary Definitions}
In this section, we will discuss many of the definitions and ideas which serve as the foundation for this article.  We begin by summarizing many of the factorization definitions from \cite{Valdezleon} in which they study various types of associate relations and irreducible elements.  We then define several finite factorization properties that a ring may possess based upon different choices of irreducible and associate.  We will also provide many of the requisite definitions regarding irreducible divisor graphs especially from \cite{Axtellidgd} and \cite{Coykendall}.  This will allow us to define a number of graphs associated with a particular commutative ring with $1\neq 0$.

\subsection{Factorization Definitions in Rings with Zero-Divisors}\ 
\\
\indent As in \cite{Valdezleon}, we let $a \sim b$ if $(a)=(b)$, $a\approx b$ if there exists $\lambda \in U(R)$ such that $a=\lambda b$, and $a\cong b$ if (1) $a\sim b$ and (2) $a=b=0$ or if $a=rb$ for some $r\in R$ then $r\in U(R)$.  We say $a$ and $b$ are \emph{associates} (resp. \emph{strong associates, very strong associates}) if $a\sim b$ (resp. $a\approx b$, $a \cong b$).  As in \cite{Stickles}, a ring $R$ is said to be \emph{strongly associate} (resp. \emph{very strongly associate}) ring if for any $a,b \in R$, $a\sim b$ implies $a \approx b$ (resp. $a \cong b$).
\\
\indent This leads to several different types of irreducible elements and we refer the reader to \cite[Section 2]{Valdezleon} for more equivalent definitions of the following irreducible elements.  A non-unit $a\in R$ is said to be \emph{irreducible or atomic} if $a=bc$ implies $a\sim b$ or $a\sim c$.  A non-unit $a\in R$ is said to be \emph{strongly irreducible or strongly atomic} if $a=bc$ implies $a \approx b$ or $a\approx c$.  A non-unit $a\in R$ is said to be \emph{m-irreducible or m-atomic} if $a$ is maximal in the set of proper principal ideals of $R$.  A non-unit $a\in R$ is said to be \emph{very strongly irreducible or very strongly atomic} if $a=bc$ implies that $a\cong b$ or $a \cong c$.  We retain the usual definition of a prime element, where $a\in R$ is said to be \emph{prime or p-atomic} if $a\mid bc$ implies $a\mid b$ or $a\mid c$.
\\

\indent We have the following relationship between the various types of irreducibles which is proved in \cite[Theorem 2.13]{Valdezleon}.

\begin{theorem}\label{thm: unrefinable} Let $R$ be a commutative ring with $1$ and let $a \in R$ be a non-unit.  The following diagram illustrates the relationship between the various types of irreducibles $a$ might satisfy.

$$\xymatrix{
\text{very strongly irreducible} \ar@{=>}[r] & \text{m-irreducible} \ar@{=>}[r] & \text{strongly irreducible} \ar@{=>}[r] & \text{irreducible} \\
 & & & \text{prime} \ar@{=>}[u]  }$$
\end{theorem}

\indent Following A. Bouvier, a ring $R$ is said to be \emph{pr\'esimplifiable} if $x=xy$ implies $x=0$ or $y\in U(R)$ as in \cite{bouvier71, bouvier72a, bouvier72b, bouvier74, bouvier74b}.  When $R$ is pr\'esimplifiable, the various associate relations coincide.  If $R$ is pr\'esimplifiable, then irreducible will imply very strongly irreducible and the various types of irreducible elements will also coincide.  Prime remains strictly stronger than irreducible even in the case of integral domains.  Any integral domain or quasi-local ring is pr\'esimplifiable.  Examples are given in \cite{Valdezleon} and abound in the literature which show that in a general commutative ring setting, each of these types of irreducible elements are distinct. 
\\
\indent This yields the following finite factorization properties that a ring may possess.  Let $\alpha \in \{$atomic, strongly atomic, m-atomic, very strongly atomic$ \}$, $\beta \in \{$associate, strong associate, very strong associate$\}$.  Then $R$ is said to be \emph{$\alpha$} if every non-unit $a\in R$ has a factorization $a=a_1\cdots a_n$ with $a_i$ being $\alpha$ for all $1\leq i \leq n$.  We will call such a factorization a \emph{$\alpha$-factorization}.  We say $R$ satisfies the \emph{ascending chain condition on principal ideals (ACCP)} if for every chain $(a_0) \subseteq (a_1) \subseteq \cdots \subseteq (a_i) \subseteq \cdots$, there exists an $N\in \N$ such that $(a_i)=(a_N)$ for all $i>N$.
\\
\indent A ring $R$ is said to be a \emph{$\alpha$-$\beta$-unique factorization ring ($\alpha$-$\beta$-UFR)} if (1) $R$ is $\alpha$ and (2) for every non-unit $a \in R$ any two $\alpha$ factorizations $a_1 a_1 \cdots a_n = b_1 \cdots b_m$ have $m=n$ and there is a rearrangement so that $a_i$ and $b_i$ are $\beta$.  A ring $R$ is said to be a \emph{$\alpha$-half factorization ring or half factorial ring ($\alpha$-HFR)} if (1) $R$ is $\alpha$ and (2) for every non-unit $a \in R$ any two $\alpha$-factorizations have the same length.  A ring $R$ is said to be a \emph{bounded factorization ring (BFR)} if for every non-unit $a \in R$, there exists a natural number $N(a)$ such that for any factorization $a=a_1 \cdots a_n$, $n \leq N(a)$. A ring $R$ is said to be a \emph{$\beta$-finite factorization ring ($\beta$-FFR)} if for every non-unit $a \in R$ there are only a finite number of factorizations up to rearrangement and $\beta$.  A ring $R$ is said to be a \emph{$\beta$-weak finite factorization ring ($\beta$-WFFR)} if for every non-unit $a \in R$, there are only finitely many $b\in R$ such that $b$ is a divisor of $a$ up to $\beta$.  A ring $R$ is said to be a \emph{$\alpha$-$\beta$-divisor finite ring ($\alpha$-$\beta$-df ring)} if for every non-unit $a \in R$, there are only finitely many $\alpha$ divisors of $a$ up to $\beta$.
\\
\indent We will also find occasion to be interested in the following definitions, where we consider factorizations distinct if they include different ring elements, i.e. not necessarily only up to associate of some type.  A ring $R$ is said to be a \emph{strong-finite factorization ring (strong-FFR)} if for every non-unit $a \in R$ there are only a finite number of factorizations up to rearrangement.  A ring $R$ is said to be a \emph{strong-weak finite factorization ring (strong-WFFR)} if for every non-unit $a \in R$, there are only finitely many divisors of $a$.  A ring $R$ is said to be a \emph{strong-$\alpha$-divisor finite ring (strong-$\alpha$-df ring)} if for every non-unit $a \in R$, there are only finitely many $\alpha$ divisors of $a$.
\\
\indent We have the following relationships between the above properties as proved in \cite{Valdezleon} or by using $\tau$-factorization in \cite{Mooney} with $\tau=R^{\#}\times R^{\#}$ (which is associate preserving and refinable) where we get the usual factorization.  We summarize these relationships by way of the following diagram accompanying \cite[Theorem 4.1]{Mooney}.
$$\xymatrix{
            &        \alpha \text{-HFR} \ar@{=>}[dr]     &             &                  &                 \\
\alpha\text{-} \beta \text{-UFR} \ar@{=>}[ur] \ar@{=>}[r]  & \beta \text{-FFR} \ar@{=>}[r] \ar@{=>}[d]  & \text{BFR} \ar@{=>}[r]& \text{ACCP} \ar@{=>}[r]& \alpha\\
            & \beta\text{-WFFR} \ar@{=>}[d] \ar@{=>}[dl]                  &              &  \text{ACCP} \ar@{=>}[u]               &                  \\
\alpha\  \alpha\text{-}\beta \text{-df ring} \ar@{=>}[r]           & \alpha\text{-}\beta \text{-df ring} &
            }$$

\subsection{Irreducible Divisor Graph Definitions} \label{subsec: graph theory}\ 
\\
\indent We begin with some definitions from M. Axtell, N. Baeth, and J. Stickles in \cite{Axtellidgzd}.  In this paper, the authors let $Irr(R)$ be the set of all irreducible elements in a ring $R$.  
Then $\overline{Irr}(R)$ is a (pre-chosen) set of coset representatives of the collection $\{a U(D) \mid a \in Irr(D)\}$.  Let $x\in D^{\#}$ have a factorization into irreducibles.  
The irreducible divisor graph of $x \in D^{\#}$, will be the graph $G(x)=(V,E)$ where $V=\{a\in \overline{Irr}(D) \mid \ a \vert x\}$, i.e. the set of irreducible divisors of $x$ up to associate.  Given $a_1, a_2 \in \overline{Irr}(D)$, $a_1a_2 \in E$ if and only if $a_1a_2 \mid x$.  Furthermore, $n-1$ loops will be attached to $a$ if $a^n \mid x$.  If arbitrarily many powers of $a$ divide $x$, we allow an infinite number of loops.  They define the \emph{reduced irreducible divisor graph} of $x$ to be the subgraph of $G(x)$ which is formed by deleting all the loops and denote it as $\overline{G}(x)$.  A \emph{clique} will refer to a simple (no loops or multiple edges), complete (all vertices are pairwise adjacent) graph.  A clique on $n \in \N$ vertices will be denoted $K_n$.  We will call a graph $G$ a \emph{pseudo-clique} if $G$ is a complete graph having some number of loops (possibly zero).  This means a clique is a pseudo-clique and the reduced graph of a pseudo-clique is a clique.
\\
\indent Let $G$ be a graph, possibly with loops.  Let $a\in V(G)$, then we have two ways of counting the degree of this vertex.  We define \emph{deg}$(a):= \left|\{a_1 \in V(G) \mid a_1 \neq a, a_1a \in E(G)\}\right|$, i.e. the number of distinct vertices adjacent to $a$.  Suppose a vertex $a$ has $n$ loops, then we define \emph{degl}$(a):=n+deg(a)$, the sum of the degree and the number of loops.  Given $a,b \in V(G)$, we define $d(a,b)$ to be the shortest path between $a$ and $b$.  If no such path exists, i.e. $a$ and $b$ are in disconnected components of $G$, or the shortest path is infinite, then we say $d(a,b)= \infty$.  We define Diam$(G):=$sup$(\{d(a,b) \mid a,b \in V(G) \})$.  
\\
\indent Two other numbers that we will be interested for their relationship with lengths of factorizations will be the \emph{clique number} and what we call the \emph{pseudo-clique number}.  The \emph{clique number}, written $\omega(G)$, is the cardinality of the vertex set of the largest complete subgraph contained in $G$.  If for all $n \geq 2$, there is a subgraph isomorphic to $K_n$, the complete graph on $n$ vertices, then we say $\omega(G)=\infty$.  We define the \emph{pseudo-clique number} of a pseudo-clique to be the cardinality of the edge set, including loops, in a pseudo-clique.  The \emph{pseudo-clique number} of an arbitrary graph $G$, written $\Omega(G)$, will be the cardinality of the edge set of the largest pseudo-clique appearing as a subgraph of $G$.  If there are pseudo-cliques with arbitrarily many edges or loops, we say $\Omega(G)=\infty.$
\\
\indent A major obstacle to studying factorization properties in rings with zero-divisors is that there are several choices to make for associate relations as well as several choices of irreducible elements.  In \cite{Axtellidgzd}, the authors choose one particular type of irreducible and one choice for associate.  To make our results as general as possible, we will consider several possible irreducible graphs which use various choices for associate relations as well as different types of irreducible elements.  This choice makes matters somewhat more complicated; however, it will allow us to prove several equivalences with the various choices of finite factorization properties that rings may possess from \cite{Valdezleon} and elsewhere in the literature.

\section{Irreducible Divisor Graph Definitions and Relationships}
Let $R$ be a commutative ring with $1$, let $\alpha \in \{ \emptyset, $ prime, irreducible, strongly irreducible, m-irreducible, very strongly irreducible $\}$ and let $\beta \in \{ \emptyset, $ associate, strong associate, very strong associate $ \}$.  The notation when $\alpha$ or $\beta$ is $\emptyset$ is to indicate a blank space in the following irreducible divisor graph notation and should make sense in context.  
\\
We let $ A_\alpha(R)= \{ a \in R-U(R) \mid a \text{ is } \alpha \}.$  When $\alpha = \emptyset$, $A_\emptyset (R)=A(R)=R-U(R)$.  We will let $A_\alpha^\beta(R)$ be the set where we select a representative of $A_\alpha$ up to $\beta$.  If $\beta= \emptyset$, then we do not eliminate any elements from $A_\alpha(R)$.  That is, each element is represented on its own and $A_\alpha^\emptyset(R)=A_\alpha(R)$.  If $\alpha = \beta = \emptyset$, $A_\emptyset^\emptyset(R) = A(R) = R-U(R)$.
\\
\indent Now, let $x\in R$ be a non-unit.  We are now ready to define $G_\alpha^\beta(x)$, the $\alpha$-$\beta$-divisor graph of $x$.  We have the vertex set defined by $V(G_\alpha^\beta(x))=\{a\in A_\alpha^\beta(R) \vert a\mid x \}$.  The edge set is given by $ab \in E(G_\alpha^\beta(x))$ if and only if $a,b \in V(G_\alpha^\beta(x))$ and there is a $\alpha$-factorization of the form $x=ab a_1 \cdots a_n$ (if $\alpha = \emptyset$, this need only be an ordinary factorization).  Furthermore, $n-1$ loops will be attached to the vertex corresponding to $a$ if there is a $\alpha$-factorization of the form $x=a^n a_1 \cdots a_n$.  We allow for the possibility for an infinite number of loops if arbitrarily large powers of $a$ divide $x$.  
\\
\begin{lemma}\label{lem: irreducible graphs by atom} Let $R$ be a commutative ring and let $x\in R$ be a non-unit.  We fix a $\beta \in \{ \emptyset, $ associate, strong associate, very strong associate $ \}$.  We consider the following possible $\alpha$-$\beta$ divisor graphs of $x$.
\begin{enumerate}
\item $G_\emptyset^\beta(x)$
\item $G_\text{prime}^\beta(x)$
\item $G_\text{irred.}^\beta(x)$
\item $G_\text{s. irred.}^\beta(x)$
\item $G_\text{m-irred.}^\beta(x)$
\item $G_\text{v.s. irred.}^\beta(x)$
\end{enumerate}
\indent Then we have the following inclusions between the graphs, i.e. the graph appears as a subgraph.

$$\xymatrix{G_\text{v.s. irred.}^\beta(x) \ar@{^{(}->}[r] & G_\text{m-irred.}^\beta(x)\ar@{^{(}->}[r] & G_\text{s. irred.}^\beta(x) \ar@{^{(}->}[r] & G_\text{irred.}^\beta(x)\ar@{^{(}->}[r]& G_\emptyset^\beta(x) \\
                                              &                               &                     &  G_\text{prime}^\beta(x) \ar@{^{(}->}[u]  &     }$$																							
																							
\end{lemma}
\begin{proof}Once we have fixed the representative of the associate classes up to $\beta$, we may apply Theorem \ref{thm: unrefinable} to see that the vertex set containments agree.  All the very strongly irreducible elements are m-irreducible which are strongly irreducible which are irreducible giving us the horizontal inclusions.  Lastly, we know that the prime elements of a ring are among the irreducible elements, which demonstrates the vertical inclusion.  Hence the vertex sets satisfy the relationships described in the diagram.
\\
\indent We now let $\alpha$ be the appropriate type of irreducible or prime in the graph we wish to show is included and let $\alpha'$ be the type of irreducible or prime we wish to show contains the given edge.  Let $a_1a_2\in E(G_\alpha^\beta(x))$. Then there is a factorization of the form $x=a_1 \cdots a_n$ where $a_i$ is $\alpha$ for each $1\leq i \leq n$.  If $\alpha_i$ is $\alpha$, then it is also $\alpha'$, so this factorization is also a $\alpha'$-factorization by Theorem \ref{thm: unrefinable}.  This proves that $a_1a_2\in E(G_{\alpha'}^\beta(x))$ as desired.
\end{proof}
\begin{lemma}\label{lem: irreducible graphs by associate} Let $R$ be a commutative ring and let $x\in R$ be a non-unit.  We fix a $\alpha \in \{ \emptyset, $ prime, irreducible, strongly irreducible, m-irreducible, unrefinably irreducible, very strongly irreducible $ \}$.  We consider the following possible $\alpha$-$\beta$ divisor graphs of $x$.
\begin{enumerate}
\item $G^\text{associate}_\alpha(x)$
\item $G^\text{s. associate}_\alpha(x)$
\item $G^\text{v.s. associate}_\alpha(x)$
\item $G^\emptyset_\alpha(x)$
\end{enumerate}
\indent We use the symbol $\xymatrix{ G_1 \ar@{^{(}->}^{\sim}[r] & G_2}$ to denote that $G_1$ is a quotient graph of $G_2$, where vertices in $G_2$ have been identified with each other and consolidated into one vertex in $G_1$.  Any edges between identified vertices from $G_2$ are now loops in $G_1$. Then we have the following inclusions between the graphs.
$$\xymatrix{ G^\text{associate}_\alpha(x) \ar@{^{(}->}^{\sim}[r] & G^\text{s.assoc.}_\alpha(x) \ar@{^{(}->}^{\sim}[r] & G^\text{v.s.assoc.}_\alpha(x) \ar@{^{(}->}^{\sim}[r] & G^\emptyset_\alpha(x) }$$
\end{lemma}
\begin{proof} This is due to the fact that
$$\emptyset \subseteq \{(a,b)\in R^{\#} \times R^{\#} \mid a\cong b\} \subseteq \{(a,b)\in R^{\#} \times R^{\#} \mid a\approx b\} \subseteq \{(a,b)\in R^{\#} \times R^{\#} \mid a\sim b\}.$$  As we go from right to left, we see more vertices get identified together as we move from stronger forms of associate to a weaker form of associate.  To see how edges could become loops, consider $\alpha$ elements, $b,c \in R$ such that $bc \mid x$ where $b$ and $c$ are associates, but not strong associates.  Then $bc$ is a simple edge in $G^\text{s. associate}_\alpha(x)$, but it yields a loop in $G^\text{associate}_\alpha(x)$.  Analogous arguments show the rest of the inclusions.
\end{proof}
\begin{corollary}\label{cor: big one} Let $R$ be a commutative ring.  For a given non-unit $x\in R$, we have the following diagram which demonstrates the relations between the various irreducible divisor graphs of $x$.

$$\xymatrix{ G^\text{assoc.}_\text{v.s. irred.}(x) \ar@{^{(}->}^{\sim}[r]\ar@{^{(}->}[d] & G^\text{s. assoc.}_\text{v.s. irred.}(x) \ar@{^{(}->}^{\sim}[r] \ar@{^{(}->}[d]& G^\text{v.s. assoc.}_\text{v.s. irred.}(x) \ar@{^{(}->}^{\sim}[r] \ar@{^{(}->}[d] & G^\emptyset_\text{v.s. irred.}(x)\ar@{^{(}->}[d]\\
G^\text{assoc.}_\text{m-irred.}(x) \ar@{^{(}->}^{\sim}[r]\ar@{^{(}->}[d] & G^\text{s. assoc.}_\text{m-irred.}(x) \ar@{^{(}->}^{\sim}[r] \ar@{^{(}->}[d]& G^\text{v.s.assoc.}_\text{m-irred.}(x) \ar@{^{(}->}^{\sim}[r] \ar@{^{(}->}[d] & G^\emptyset_\text{m-irred.}(x)\ar@{^{(}->}[d]\\
G^\text{assoc.}_\text{s. irred.}(x) \ar@{^{(}->}^{\sim}[r]\ar@{^{(}->}[d] & G^\text{s. assoc.}_\text{s. irred.}(x) \ar@{^{(}->}^{\sim}[r] \ar@{^{(}->}[d]& G^\text{v.s. assoc.}_\text{s. irred.}(x) \ar@{^{(}->}^{\sim}[r] \ar@{^{(}->}[d] & G^\emptyset_\text{s. irred.}(x)\ar@{^{(}->}[d]\\
G^\text{assoc.}_\text{irred.}(x) \ar@{^{(}->}^{\sim}[r]\ar@{^{(}->}[d] & G^\text{s. assoc.}_\text{irred.}(x) \ar@{^{(}->}^{\sim}[r] \ar@{^{(}->}[d]& G^\text{v.s. assoc.}_\text{irred.}(x) \ar@{^{(}->}^{\sim}[r] \ar@{^{(}->}[d] & G^\emptyset_\text{irred.}(x)\ar@{^{(}->}[d]\\
G^\text{assoc.}_\emptyset(x) \ar@{^{(}->}^{\sim}[r] & G^\text{s. assoc.}_\emptyset(x) \ar@{^{(}->}^{\sim}[r] & G^\text{v.s. assoc.}_\emptyset(x) \ar@{^{(}->}^{\sim}[r] & G^\emptyset_\emptyset(x)\\
G^\text{assoc.}_\text{prime}(x) \ar@{^{(}->}[u] \ar@{^{(}->}^{\sim}[r]\ar@/^3pc/[uu]  & G^\text{s. assoc.}_\text{prime}(x) \ar@{^{(}->}[u] \ar@{^{(}->}^{\sim}[r] \ar@/^3pc/[uu]& G^\text{v.s. assoc.}_\text{prime}(x) \ar@{^{(}->}[u] \ar@{^{(}->}^{\sim}[r] \ar@/_3pc/[uu] & G^\emptyset_\text{prime}(x) \ar@{^{(}->}[u] \ar@/_3pc/[uu]
}$$

\end{corollary}
\begin{proof}This is an immediate corollary of Lemma \ref{lem: irreducible graphs by atom} and Lemma \ref{lem: irreducible graphs by associate}.
\end{proof}
\indent The following theorems indicate certain situations in which many of the associate relations and irreducibles would coincide. 
\begin{theorem}(\cite[Theorem 2.13]{Valdezleon}) $0$ is m-irreducible if and only if $R$ is a field.  $R$ is a domain if and only if $0$ is irreducible if and only if $0$ is prime if and only if $0$ is strongly irreducible if and only if $0$ is very strongly irreducible.
\end{theorem}

\begin{theorem}\label{thm: presimplifiable} Let $R$ be a commutative ring with $1$.  If $R$ is pr\'esimplifiable, $x\in R$ is a non-zero, non-unit, $\alpha \in \{$ irreducible, strongly irreducible, m-irreducible, very strongly irreducible $\}$ and $\beta \in \{$ associate, strong associate, very strong associate $\}$, $G_\alpha^{\beta}$(x) is the same for any choice of $\alpha$ and $\beta$ provided the same choice of $\beta$ representative is selected.
\end{theorem}
\begin{proof} As discussed in the preliminaries, in a pr\'esimplifiable ring $a \sim b$ if and only if $a \approx b$ if and only if $a\cong b$.  This implies that $x\in R^{\#}$ is atomic if and only if $x$ is strongly atomic if and only if $x$ is m-atomic if and only if $x$ is very strongly atomic.  This shows the choices of irreducible and associate all coincide, so their respective irreducible divisor graphs will also coincide.
\end{proof}

\begin{remark} The reader may be wondering about why prime no longer fits into the above theorem.  Even in domains, which are certainly pr\'esimplifiable, there are examples of irreducible elements which are not prime.  For instance $9 \in \Z[\sqrt{-5}]$ has irreducible factorizations $9=3\cdot 3= (2 + \sqrt{-5})(2 - \sqrt{-5})$; however, these are not prime factorizations.  Because of this, we focus on irreducible elements and irreducible factorizations throughout the rest of the paper.
\end{remark}

\section{Irreducible Divisor Graphs and Irreducible Elements}
\indent An interesting thing to note was that in the domain case, if $x\in D^{\#}$ is irreducible, then $G(x)\cong K_1$, a single vertex.  In an integral domain, the only factorizations of an irreducible element $x$ are trivial factorizations of the form $x=\lambda (\lambda^{-1}x)$.  This is not necessarily the case when there are zero-divisors present.  With this in mind, we present following example and use this to motivate the investigation of this more thoroughly throughout the rest of the section.

\begin{example}  Let $R= \Z \times \Z$.  
\\
\indent We consider the element $(1,0)$ and consider what possible factorizations could look like.  If $(1,0)=(a_1,b_1)(a_2,b_2)\cdots (a_n,b_n)$, then it must be the case that $a_1 a_2 \cdots a_n=1$ and $b_1b_2\cdots b_n=0$.  The fact that $a_1a_2 \cdots a_n=1$ implies that the first coordinate of any factor in a factorization of $(1,0)$ must be a unit.  The fact that $\Z$ is an integral domain and $b_1 b_2 \cdots b_n=0$ implies that in any factorization of $(1,0)$ at least one factor must have a zero in the second coordinate.  Thus any factorization of $(1,0)$ must have $(1,0)$ or $(-1,0)$ occurring somewhere in the factorization.  This demonstrates that $(1,0)$ is both irreducible and strongly irreducible.  
\\
\indent On the other hand, $(1,2) \mid (1,0)$ as seen by $(1,0)=(1,2)(1,0)$; however, it is clear that $(1,0)$ cannot divide $(1,2)$ due to the second coordinate being non-zero.  This demonstrates that $(1,0) \subsetneq (1,2)$ which in turn shows that $(1,0)$ is not m-atomic.  Moreover $(1,2)$ is not a unit and hence $(1,0)=(1,2)(1,0)$ demonstrates $(1,0)$ is not very strongly atomic either.  
\\
\indent We are now interested in what other types of irreducible elements divide $(1,0)$.  A non-zero, non-unit element in $\Z \times \Z$ is irreducible if and only if it is of the form $(\pm 1, p)$ or $(p, \pm 1)$ with $p$ an irreducible element of $\Z$.  In a UFD like $\Z$, non-zero prime elements and irreducible elements coincide.  We also note that in a domain $0$ is irreducible since there are no non-trivial zero-divisors.  
\\
\indent Elements of the form $(1, p)$ for $p$, a non-zero irreducible, are regular elements and therefore all of the notions of irreducible will coincide.  Thus $(1,p)$ for $p$ a non-zero irreducible is irreducible, strongly irreducible, m-irreducible, and very strongly irreducible.  We will go ahead and choose the positive values of $p$ as our equivalence class representatives.  Hence all irreducible, and strongly irreducible factorizations of $(1,0)$ up to associate (and strongly associate since $\Z \times \Z$ is a strongly associate ring) are of the form 
$$(1,0)=(1,0)^{i_0}(1,p_1)^{i_1}(1,p_2)^{i_2}\cdots(1,p_n)^{i_n}$$
where $p_i$ is a non-zero irreducible element for each $1 \leq i \leq n$.  Hence when studying the irreducible and strongly irreducible divisor graph of $(1,0)$ up to associate and strong associates, we get a complete graph on an infinite number of vertices generated by elements $\{(1,p)\mid \text{ p is non-negative and irreducible in } \Z \}$.  Moreover, each vertex has an infinite number of loops.
\\
\indent When considering factorizations up to very strongly associate, we must be slightly careful because $(1,0) \not \cong (-1,0)$, so we actually will need to consider atomic and strongly atomic factorizations of the form
$$(1,0)=(-1,0)^{i'_{0}}(1,0)^{i_0}(1,p_1)^{i_1}(1,p_2)^{i_2}\cdots(1,p_n)^{i_n}$$
where $p_i$ is a non-zero irreducible element for each $1 \leq i \leq n$.  Hence we get a complete graph on an infinite number of vertices generated by elements 
$$\{(1,p)\mid \text{ p is non-negative and irreducible in } \Z \}\cup \{(-1,0)\}.$$ 
Again each vertex will have an infinite number of loops.
\\

\indent This shows that for $\alpha\in \{$ irreducible, strongly irreducible $\}$ and for $\beta \in \{$ associate, strongly associate $\}$, we have the following for $G_\alpha^\beta\left( (1,0)\right)$.
\begin{figure}[H]
	 \centering
		 \includegraphics[scale=.7]{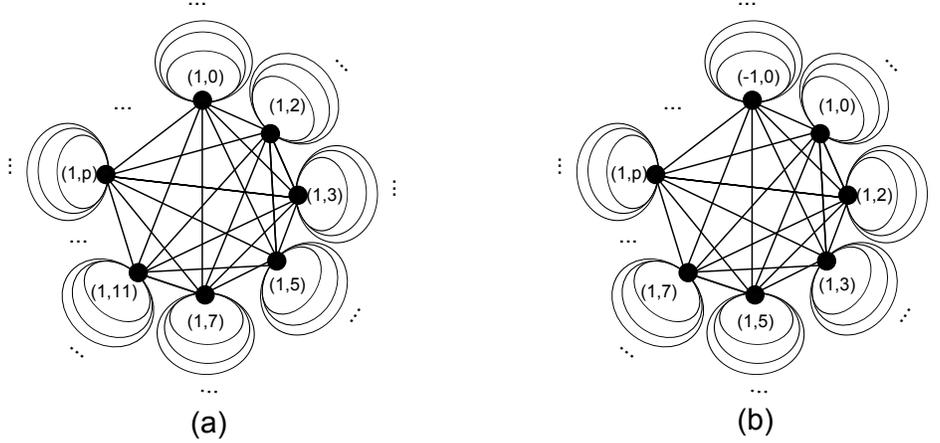}
	 \caption{(a) $G_\alpha^\beta \left( (1,0)\right)$\ \ \ (b) $G_\alpha^\text{v.s. associate} \left( (1,0)\right)$}
	 \label{fig:Irreducible divisor graph very strongly associate}
 \end{figure}
\indent We now turn our attention to the divisor graphs, where we do not restrict the factors to types of irreducibles, but instead allow any divisors of $(1,0)$.  The factorizations come in the form
$$(1,0)=(1,0)^{i_0}(1,n_1)^{i_1}(1,n_2)^{i_2}\cdots(1,n_m)^{i_m}$$
where $n_i$ is some non-unit, positive, natural number for $1 \leq i \leq m$.  If we are looking up to very strongly associate, then we need to again allow factorizations of the form:
$$(1,0)=(-1,0)^{i'_{0}}(1,0)^{i_0}(1,n_1)^{i_1}(1,n_2)^{i_2}\cdots(1,n_m)^{i_m}.$$
\indent When we choose associate or strong associate, we again get a complete graph on an infinite number of vertices generated by elements 
$$\{(1,n)\mid \text{ n is non-negative integer, but not 1 }\}$$ 
with each vertex having an infinite number of loops.  When we choose very strong associate we have the vertex set
$$\{(1,n)\mid \text{ n is non-negative integer, but not 1 }\}\cup \{(-1,0)\}.$$  
\indent Hence for $\beta \in \{$ associate, strong associate $\}$, we get the following divisor graphs
\begin{figure}[H]
	 \centering
		 \includegraphics[scale=.7]{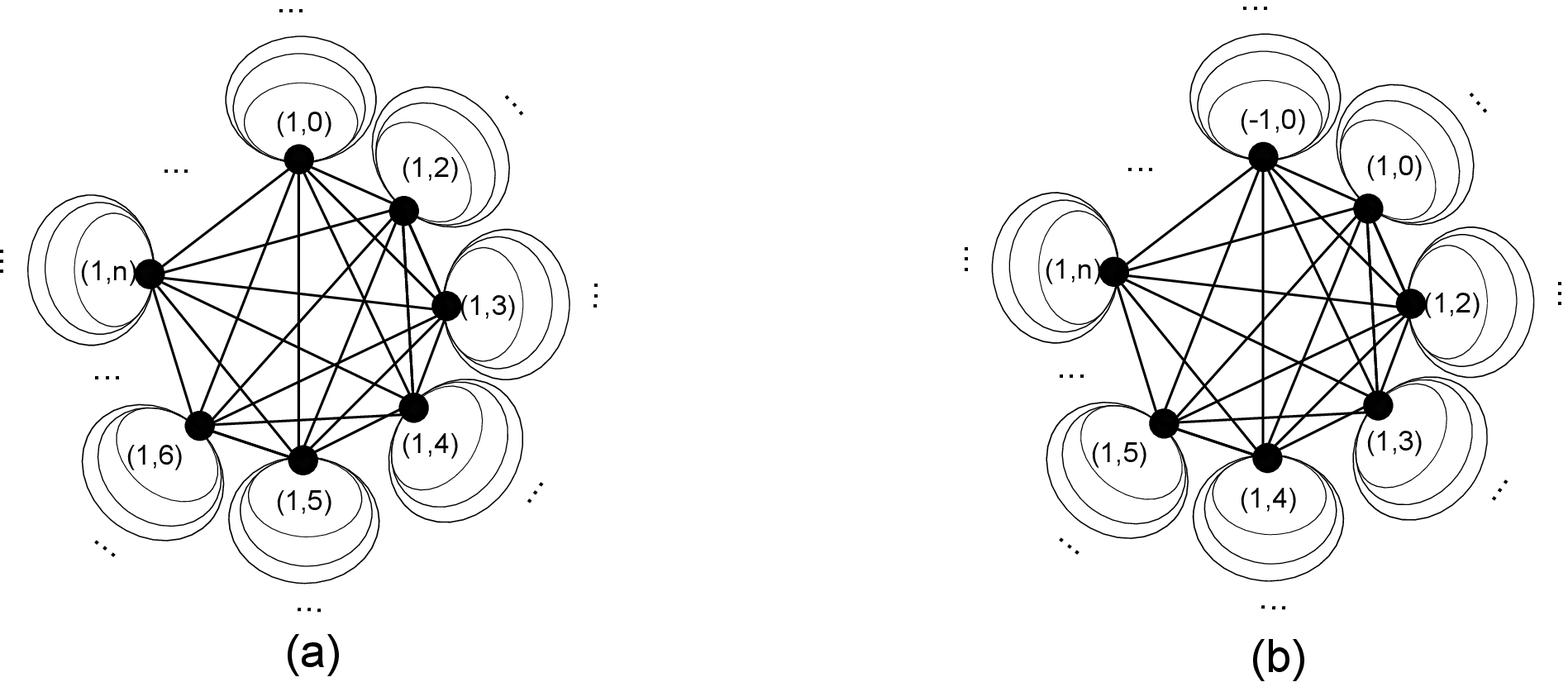}
	 \caption{(a) $G_\emptyset^\beta \left( (1,0)\right)$ \ \ \ (b) $G_\emptyset^\text{v.s. associate} \left( (1,0)\right)$}
	 \label{fig:Irreducible divisor graph no atoms}
 \end{figure}

\indent The last group of factorizations to consider will be the m-irreducible, and very strongly irreducible factorizations.  We know that the vertex set will be $\{(1,p) \mid p \text{ is a non-zero prime } \}$, so $(1,0)$ is no longer among these as demonstrated above.  Furthermore, we have seen that to successfully have a factorization of $(1,0)$ it is necessary for $(1,0)$ to occur as a factor.  This is not even a m-irreducible element, so there are can be no non-trivial m-irreducible or very strongly irreducible factorizations of $(1,0)$ and hence no edges between vertices or loops on any vertex.  Lastly, since all of these elements are regular, all of the associate relations coincide.
\\
\indent This means for $\alpha \in \{$ m-atomic, very strongly atomic $\}$ and $\beta \in \{$ associate, strongly associate, very strongly associate $\}$ we have the following for $G_\alpha^\beta \left( (1,0) \right)$.
\begin{figure}[H]
	 \centering
		 \includegraphics[scale=.7]{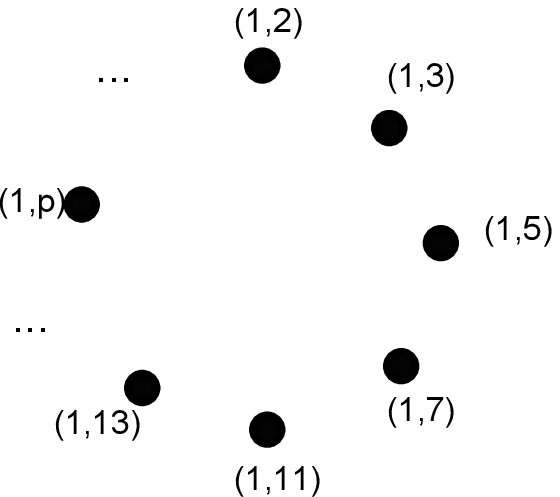}
	 \caption{$G_\alpha^\beta \left( (1,0)\right)$}
	 \label{fig:Irreducible divisor graph m-atomic}
 \end{figure}
\end{example}
\begin{remark} It is clear that none of these graphs are equal; however, the first four are certainly all isomorphic while the last one is completely disconnected.  This example also serves as a demonstration that many of inclusions suggested in Corollary \ref{cor: big one} are indeed strict.  Moreover, this example demonstrates that even for a strongly associate commutative ring with zero-divisors, $\Z \times \Z$, the irreducible divisor graph of irreducible and strongly irreducible elements can be quite complicated compared to the irreducible elements in the domain case.
\\

\indent The main issue above is that $(1,0)$ was not even m-irreducible.  It appears that for divisor graphs in rings with zero-divisors, irreducible and strongly irreducible is not quite powerful enough to get analogous results to the domain case.  To be a bit more optimistic, we do have several nice characterizations regarding the divisor graphs of the stronger choices for irreducible: m-irreducible and very strongly irreducible contained in the following theorems.
\end{remark}
\begin{theorem}  Let $R$ be a commutative ring.  If $x\in R$ is very strongly atomic, then we have the following.
\begin{enumerate}
\item $G_\emptyset^\text{strongly associate}(x)\cong K_1$, i.e. is a graph with one vertex and no loops. 
\item $G_\emptyset^\text{associate}(x)\cong K_1$. 
\item $G_\emptyset^\emptyset(x)$ is a collection of $\vert U(R) \vert$ totally disconnected vertices of the form $\{ \lambda x \mid \lambda \in U(R) \}$.
\end{enumerate}
\end{theorem}
\begin{proof} (1) There are only trivial factorizations of $x$, so all factorizations are of the form $x=\lambda (\lambda^{-1}x)$ for a unit $\lambda \in U(R)$.  But this means all divisors of $x$ are strong associates of $x$.  This proves there can be only one vertex in $G_\emptyset^\text{strongly associate}(x)$.  If there were a loop, then we would have some $a\in R^{\#}$ such that $a^2 \mid x$, but this would imply $x=a \cdot a \cdot a_1 \cdots a_n$ is a factorization of length at least $2$, contradicting the fact that $x$ is very strongly atomic.  
\\
\indent (2) By Lemma \ref{lem: irreducible graphs by associate} since $G_\emptyset^\text{associate}(x)$ is a subgraph of $G_\emptyset^\text{strongly associate}(x)$ which is a single vertex with no loops and the fact that $x=1\cdot x$ is certainly a factorization, so $G_\emptyset^\text{associate}(x)$ is non-empty.  
\\
\indent (3) This follows from the assertion previously that all divisors of $x$ are strong associates of $x$ so they are unit multiples of $x$.  Hence the number of divisors of $x$ is precisely the number of units in $R$.  Because there are no non-trivial factorizations of $x$, there can be no edges in the $G_\emptyset^\emptyset(x)$ and therefore must be totally disconnected.
\end{proof}

\begin{theorem} Let $R$ be a commutative ring.  If $x\in R$ is m-atomic, then $\overline{G}_\emptyset^\text{associate}(x)\cong K_1$, i.e. $G_\emptyset^\text{associate}(x)$ is a graph with one vertex and possibly some loops.  
\end{theorem}
\begin{proof} Clearly, if $x$ is m-atomic, then $x=1\cdot x$ is a m-atomic factorization, which implies that $x \in V(\overline{G}_\emptyset^\text{associate}(x))$.  Suppose there is another vertex, say $y$.  Hence $y$ occurs as a factor in a factorization of $x$.  Suppose $x=y a_1 \cdots a_n$ is such a factorization.  Then since $x$ is m-atomic, we know that every divisor of $x$ is associate to $x$, proving the theorem.  In particular, $x \sim y$ and they are represented by the same vertex, but could possibly contribute a loop to the graph if the factorization is non-trivial. 
\end{proof}
The following gives a converse to the previous theorems.
\begin{theorem} Let $R$ be a commutative ring.  We have the following.
\begin{enumerate}
\item If $x\in R$ is a non-unit such that $G_\emptyset^\text{strongly associate}(x)\cong K_1$, then $x$ is very strongly atomic. 
\item If $x\in R$ is a non-unit such that $E(G_\emptyset^\emptyset(x))=\emptyset$, then $x$ is very strongly atomic.
\item If there is a non-unit $x\in R$ such that $\overline{G}_\emptyset^\text{associate}(x)\cong K_1$, then $x$ is m-atomic.
\end{enumerate}
\end{theorem}
\begin{proof} (1) Suppose $G_\emptyset^\text{strongly associate}(x)\cong K_1$ and $x$ were not very strongly atomic.  Let $x=a_1 \cdots a_n$ be a factorization with $n \geq 2$.  Then there is an edge in $G_\emptyset^\text{strongly associate}(x)$ between $a_1$ and $a_2$, or possibly a loop if $a_1 \approx a_2$.  Either way, it contradicts the hypothesis that $G_\emptyset^\text{strongly associate}(x)\cong K_1$.
\\
\indent (2) Let $x\in R$ be a non-unit.  Suppose $x=ab$ for some non-units $a,b \in R$.  Then $a \mid x$ and $b \mid x$, so $a, b \in V(G_\emptyset^\emptyset(x))$, possibly the same vertex.  We have $x=1\cdot ab$ which implies $ab\mid x$ showing that there is an edge (possibly a loop) between $a$ and $b$.  This is a contradiction since $E(G_\emptyset^\emptyset(x))=\emptyset$.  This proves there can be no non-trivial factorizations of $x$, making $x$ very strongly atomic as desired.
\\
\indent (3) Let $x\in R$ be a non-unit such that $\overline{G}_\emptyset^\text{associate}(x)\cong K_1$.  We suppose for a moment that $x$ were not m-irreducible.  Then there is a factorization $x=a_1 \cdots a_n$ such that there is an $a_i$ such that $x \not \sim a_i$.  But then $a_i$ is a distinct vertex in $\overline{G}_\emptyset^\text{associate}(x)$ from $x$, a contradiction of the hypothesis that $\overline{G}_\emptyset^\text{associate}(x)\cong K_1$.
\end{proof}

\begin{theorem}  Let $R$ be a commutative ring.  If $x\in R$ is atomic (resp. strongly atomic), then Diam($G_\emptyset^\text{associate}(x))$ (resp. Diam($G_\emptyset^\text{strongly associate}(x))$) is at most $2$.  Moreover, there is a vertex which is associate (resp. strongly associate) to $x$ such that every vertex is adjacent to this vertex.\end{theorem}
\begin{proof}  Let $a_1 \in V(G_\emptyset^\text{associate}(x))$ (resp. $a_1 \in V(G_\emptyset^\text{strongly associate}(x))$).  Then $a_1 \mid x$, say $x=a_1 \cdots a_n$ is a factorization.  Since $x$ is atomic (resp. strongly atomic), $x \sim a_i$ (resp. $x \approx a_i$) for some $1 \leq i \leq n$.  If $x \sim a_1$ (resp. $x\approx  a_1$), then they are in fact represented by the same vertex in the graph: whichever was chosen at the associate (resp. strong associate) class representative of $x$.  If $x \sim a_i$ (resp. $x \approx a_i$) for $2 \leq i \leq n$, say $a_i=\mu x$ for some $\mu \in R$ (resp. $a_i=\mu x$ for $\mu \in U(R)$).  Then we have a factorization
$$x=a_1a_2 \cdots a_{i-1} (\mu x) a_{i+1} \cdots a_n= \mu xa_1a_2 \cdots a_{i-1}\cdot\widehat{a_i} \cdot a_{i+1} \cdots a_n$$ 
(where $\widehat{a_i}$ indicates $a_i$ is omitted) showing $xa_1 \mid x$ and therefore $a_1$ and $x$ are adjacent as desired.  If every vertex in a graph is adjacent to a single vertex, then the diameter of the graph is certainly no larger than $2$.
\end{proof}

\section{Irreducible Divisor Graph and Finite Factorization Properties}
In this section, we investigate the relationship between finite factorization properties defined in \cite{Valdezleon} that rings may possess and characteristics of the various $\alpha$-$\beta$-irreducible divisor graphs.  
\\
\indent We begin with a remark demonstrating the relationship between factorizations of a non-unit $x\in R$ and pseudo-cliques in the divisor graph.
\begin{remark} Let $\alpha \in \{ \emptyset,$ irreducible, strongly irreducible, m-irreducible, very strongly irreducible $\}$ and let $\beta \in \{ \emptyset,$ associate, strongly associate, very strongly associate $\}$.  Let $x\in R$ be a non-unit and $x=a_1 \cdots a_n$ be a $\alpha$-factorization of $x$.  Then there is an associated pseudo-clique in $G_\alpha^\beta (x)$.  Suppose $a_1, \ldots a_s$ are distinct factors of $x$ up to $\beta$ with $1 \leq s \leq n$.  We then may rewrite the factorization in the form $x=a_1^{e_1}a_2^{e_2} \cdots a_s^{e_s}$ where $e_1 + e_2 + \cdots + e_s=n$.  Then there is a pseudo-clique subgraph in $G_\alpha^\beta (x)$ with vertex set $\{a_1, \ldots, a_s \}$ such that $a_i$ and $a_j$ are adjacent for all $i \neq j$ and $1 \leq i,j \leq s$ and $a_i$ has $e_i-1$ loops for each $1 \leq i \leq s$.  We refer to this as the \emph{subgraph associated to the factorization} and will denote it $S$.  
\\
\indent If we look at the reduced graph, $\overline{S}$, by removing the loops from $S$, we get $\overline{S} \cong K_s$.  So $\omega(S)=\omega(\overline{S})=s$.  We could also count the number of edges in $\overline{S}$, it would be $\binom{s}{2}=\frac{s(s-1)}{2}$.  On the other hand, in $R=\Z/2\Z \times \Z/2\Z$ $(1,0)=(1,0)^i$ yields arbitrarily long factorizations.  This leads to a graph with a vertex having an infinite number of loops.  It is here that we see $R$ fails to be a FFR or even a BFR.  This motivates the introduction of studying the pseudo-clique number, denoted $\Omega(S)$, of a graph rather than just the clique number.  
\\
\indent Recall from Section \ref{subsec: graph theory} that the pseudo-clique number of a graph is the number of edges and loops in the largest pseudo-clique in the graph.  A graph is said to have infinite pseudo-clique number if there are pseudo-cliques with arbitrarily many edges or loops. The pseudo-clique number of the subgraph, $S$, associated with the factorization $x= a_1 \cdots a_n= a_1^{e_1}a_2^{e_2} \cdots a_s^{e_s}$, is given by the following function
$$\phi(n,s)=\Omega(S)=\binom{s}{2} + (n-s)=\frac{s(s-1)}{2} + n - s.$$
\indent Given a factorization of length $n$, we can compute explicitly the pseudo-clique number of the associated graph as a function of $s$, the number of distinct divisors. The number of edges is maximal when each factor is distinct, and minimal when there are only one or two distinct factors.  For an $\alpha$-factorization of length $n$, as a function of $s$, the number of distinct factors, we have
$$n-1 \leq \phi(n,s) \leq \binom{n}{2}=\frac{n(n-1)}{2}.$$
These bounds are tight in the sense that they can be achieved on the low end when $s=1$ or $s=2$, $\phi(n,1)=n-1=\phi(n,2)$ and on the high end, when $s=n$, we have $\phi(n,n)=\binom{n}{2}$.

\end{remark}
\begin{theorem}\label{thm: accp} Let $R$ be a commutative ring and let $\alpha \in \{$ atomic, strongly atomic, m-atomic very strongly atomic $\}$ and let $\beta \in \{$ associate, strongly associate, very strongly associate $\}$.  If $R$ is $\alpha$ and for all $x \in R$, a non-unit, and for all $a \in V(G_\alpha^\beta(x))$, degl$(a) < \infty$, then $D$ satisfies ACCP.
\end{theorem}
\begin{proof} Suppose $R$ did not satisfy ACCP.  Then there exists a chain of principal ideals $(x_1) \subsetneq (x_2) \subsetneq (x_3) \subsetneq \cdots$ .  Say 

\begin{equation}\label{1} x_i = x_{i+1}\cdot a_{i1} \cdots a_{in_i}\end{equation}

\noindent is a factorization for each $i$.  Because $R$ is $\alpha$, we may replace each $a_{ij}$ with a $\alpha$ factorization.  This allows us to assume each factor in Equation \eqref{1} is $\alpha$.  We may assume further that each $a_{ij}$ is one of the pre-chosen $\beta$-representatives.  We may iterate these substitutions as follows 

\begin{equation}\label{2}x_1 = x_{2}\cdot a_{11} \cdots a_{1n_1}= x_{3}\cdot a_{21} \cdots a_{2n_2}\cdot a_{11} \cdots a_{1n_1}= \cdots\end{equation}

\noindent and each is a factorizations with $a_{ij}$ being $\alpha$ for all $i$ and $j$.  Because $(x_i)$ is properly contained in  $(x_{i+1})$, in Equation \eqref{1} $n_i \geq 1$ or else $x_i \sim x_{i+1}$.  This means the factorizations in each iteration of Equation \eqref{2} strictly increase in length.  If $\{a_{ij}\}$ is infinite, then $a_{11}$ has an infinite number of adjacent vertices in $V(G_\alpha^\beta (x))$, i.e $degl(a_{11}) \geq deg(a_{11})=\infty$.  Otherwise, if $\{a_{ij}\}$ is finite, then one of the $a_{i_0j_0}$ for some $i_0$ and $j_0$ occurs an infinite number of times.  Hence degl$(a_{i_0j_0})=\infty$ in $G_\alpha^\beta(x)$ since arbitrarily large powers of $a_{i_0j_0}$ divide $x_1$.  This is a contradiction and thus $R$ must satisfy ACCP as desired.
\end{proof}
\indent We could also state the previous theorem without the atomic hypothesis as follows.

\begin{theorem}  Let $R$ be a commutative ring.  If for all $x \in R$, a non-unit, and for all $a \in V(G_\emptyset^\beta(x))$, degl$(a) < \infty$, then $D$ satisfies ACCP.
\end{theorem}
\begin{proof} The proof of this is identical to \ref{thm: accp}, except we need not worry about refining the factorizations into atomic factorizations.  The rest of the argument goes through in the same fashion.
\end{proof}

\begin{theorem} \label{thm: bfr} Let $R$ be a commutative ring.  Let $\beta \in \{ \emptyset$, associate, strong associate, very strong associate $\}$ and let $x\in R$ be a non-unit.  If $G^{\beta}_{\emptyset}(x)$ has a finite pseudo-clique number, then there is a bound on the length of factorizations of $x$.  If this holds for all non-units $x\in R$, then $R$ is a BFR.
\end{theorem}
\begin{proof} Suppose $\Omega(G^{\beta}_{\emptyset}(x))=N_x<\infty$. Then by the computations done in the remarks, a factorization of length $n$, $x=a_1 \cdots a_n$, yields an associated pseudo-clique $S$ and $n-1 \leq \Omega(S) \leq N_x$.  Thus we may set $N (x)=N_x +1$ and we have found a bound on the length of any factorization of $x$.  The final statement is immediate by definition of BFR.
\end{proof}
\indent There are authors who define a BFR in terms of bounds on lengths of atomic factorizations instead.  So if $\alpha \in \{$ atomic, strongly atomic, m-atomic, very strongly atomic $\}$, then we will say that $R$ is a \emph{$\alpha$-bounded factorization ring ($\alpha$-BFR)} if for every non-unit $x\in R$, there is a bound on the length of $\alpha$-factorizations of $x$, i.e. factorizations in which every factor is $\alpha$.  It is clear that BFR the way we have defined it is stronger than $\alpha$-BFR for any choice of $\alpha$ since any $\alpha$-factorization is certainly a factorization.  It is also clear that if one assumes the ring $R$ is $\alpha$, then the two notions are equivalent.  With this in mind, we have the following theorem.

\begin{theorem} \label{thm: bfr alt} Let $R$ be a commutative ring and let $\alpha \in \{$ atomic, strongly atomic, m-atomic, very strongly atomic $\}$ and let $\beta \in \{ \emptyset$, associate, strong associate, very strong associate $\}$.  Let $x\in R$ be a non-unit.  If $G^{\beta}_{\alpha}(x)$ has a finite pseudo-clique number, then there is a bound on the length of factorizations of $x$.  If this holds for all non-units $x\in R$, then $R$ is a $\alpha$-BFR.
\end{theorem}
\begin{proof} Suppose $\Omega(G^{\beta}_{\alpha}(x))=N_x<\infty$. Then a $\alpha$-factorization of length $n$, $x=a_1 \cdots a_n$, yields an associated pseudo-clique $S$ in $G_\alpha^\beta (x)$ and $n-1 \leq \Omega(S) \leq N_x$.  Thus we may set $N(x)=N_x +1$ and we have found a bound on the length of any $\alpha$-factorization of $x$.  The final statement is immediate by definition of $\alpha$-BFR.
\end{proof}

\begin{theorem} \label{thm: ffr} Let $R$ be a commutative ring and let $\beta \in \{$ associate, strong associate, very strong associate $\}$.    Let $x\in R$ be a non-unit.  Then the following are equivalent.
\begin{enumerate}
\item $x$ has a finite number of factorizations up to rearrangement and $\beta$.
\item $\sum_{a\in V(G^{\beta}_{\emptyset}(x))}\text{degl}(a) < \infty.$ 
\item $\vert E(G^{\beta}_{\emptyset}(x))\vert < \infty$.
\end{enumerate}
\end{theorem}
\begin{proof}
(1) $\Rightarrow$ (2) We suppose $\sum_{a\in V(G^{\beta}_{\emptyset}(x))}\text{degl}(a)$ is infinite.  If $V(G^{\beta}_{\emptyset}(x))$ is infinite, then there are an infinite number of non-$\beta$ divisors of $x$ and therefore there must be an infinite number of non-$\beta$ factorizations.  This tells us that $V(G^{\beta}_{\emptyset}(x))$ must be finite.  If $V(G^{\beta}_{\emptyset}(x))$ is finite, then there must be some $a \in V(G^{\beta}_{\emptyset}(x))$ for which degl$(a)$ is infinite.  If deg$(a)$ is infinite, then there would be an infinite number of non-$\beta$ divisors adjacent to $a$, a contradiction as before since we know that $V(G^{\beta}_{\emptyset}(x))$ is finite.  This means there must be an $a\in V(G^{\beta}_{\emptyset}(x))$ for which there are an infinite number of loops.  This yields arbitrarily long factorizations of $x$ since $a^n \mid x$ for all $n \in \N$.  This gives us an infinite number of factorizations of $x$, none of which can be rearranged up to associate.  
\\
\indent For instance, $a\mid x$ implies there is a factorization of the form $x=a\cdot b_{1_1} \cdots b_{m_1}$.  Now, since arbitrarily long powers of $a$ divide $x$, $a^{m_1 +1}\mid x$.  This implies there is a factorization of the form $a \cdots a b_{1_2} \cdots b_{m_2}$ where $a$ occurs $m_1 +1$ times.  This factorization cannot be rearranged up to associates to match the first factorization of $x$ since there are more factors of $a$ than there are total factors in the first factorization.  This process can be repeated to get a sequence of factorizations of $x$ which grow properly in length.  Hence we have found an infinite number of factorizations of $x$ up to rearrangement and $\beta$, a contradiction.
\\
\indent (2) $\Rightarrow$ (3) Suppose $\sum_{a\in V(G^{\beta}_{\emptyset}(x))}\text{degl}(a) =D \in \mathbb{N}$.  Then 
$$\vert E(G^{\beta}_{\emptyset}(x))\vert = \vert E(\overline{G}^{\beta}_{\emptyset}(x))\vert + \left(\vert E(G^{\beta}_{\emptyset}(x))\vert - \vert E(\overline{G}^{\beta}_{\emptyset}(x))\vert \right) = E + L$$
where the first term, $E$ represents the number of simple edges and the second term, $L$, represents the number of loops in the graph.  Each edge in $\overline{G}^{\beta}_{\emptyset}(x))$ contributes $2$ to the sum $\sum_{a\in V(G^{\beta}_{\emptyset}(x))}\text{degl}(a)$ and each loop contributes $1$ to $\sum_{a\in V(G^{\beta}_{\emptyset}(x))}\text{degl}(a)$.  So in particular, we have
$$D=\sum_{a\in V(G^{\beta}_{\emptyset}(x))}\text{degl}(a)=2E + L \leq 2(E + L) = 2\vert E(G^{\beta}_{\emptyset}(x))\vert $$  
and
$$D=\sum_{a\in V(G^{\beta}_{\emptyset}(x))}\text{degl}(a)=2E + L \geq E + L = \vert E(G^{\beta}_{\emptyset}(x))\vert $$  
This shows $\sum_{a\in V(G^{\beta}_{\emptyset}(x))}\text{degl}(a)$ is bounded below by $\vert E(G^{\beta}_{\emptyset}(x))\vert$ and above by $2\vert E(G^{\beta}_{\emptyset}(x))\vert$, showing that if $\sum_{a\in V(G^{\beta}_{\emptyset}(x))}\text{degl}(a)$ is finite, then so too is $\vert E(G^{\beta}_{\emptyset}(x))\vert < \infty$.
\\
\indent (3) $\Rightarrow$ (1)  We begin by noticing that any factorization of $x$, $x=a_1 \cdots a_n$ corresponds to a subgraph of $G^{\beta}_{\emptyset}(x)$, in particular a pseudo-clique.  The vertices are the non-$\beta$ $a_i$ among $\{a_1, \ldots, a_n\}$ with an edge between $a_i$ and $a_j$ if they are not $\beta$.  If $a_i$ occurs $m$ times in the factorization, then there are $m-1$ loops in the subgraph graph.  By hypothesis, there are a finite number of edges in $G^{\beta}_{\emptyset}(x)$, say $N$.  Suppose there are an infinite number factorizations of $x$, none of which can be rearranged up to $\beta$.  This would correspond to an infinite number of choices for subsets of the edge set.  However, $2^N$ is finite and is the number of all possible subsets of choices of edges or loops a contradiction, completing the proof.  
\end{proof}

\begin{corollary} \label{cor: ffr} Let $R$ be a commutative ring and let $\beta \in \{$ associate, strong associate, very strong associate $\}$.  Then the following are equivalent.
\begin{enumerate}
\item $R$ is a $\beta$-FFR.
\item For all non-units, $x\in R$, we have 
$$\sum_{a\in V(G^{\beta}_{\emptyset}(x))}\text{degl}(a) < \infty.$$ 
\item For all non-units, $x\in R$, we have $$\vert E(G^{\beta}_{\emptyset}(x))\vert < \infty.$$
\end{enumerate}
\indent Furthermore, the following are also equivalent.
\begin{enumerate}
\item $R$ is a strong-FFR 
\item For all non-units, $x\in R$, we have
$$\sum_{a\in V(G^{\emptyset}_{\emptyset}(x))}\text{degl}(a) < \infty.$$
\item For all non-units, $x\in R$, we have 
$$\vert E(G^{\emptyset}_{\emptyset}(x))\vert < \infty.$$
\end{enumerate}
\end{corollary}
\begin{proof} The first set of equivalences are an immediate corollary to Theorem \ref{thm: ffr} and the definitions.  The second set of equivalences are simply the analogue for strong-FFR.  We are no longer thinking of factorizations that can be rearranged up to associate as being the same.  It can be proved in the same way as the proof of Theorem \ref{thm: ffr}, but by not looking at factorizations up to any type of associate and similarly using the graph in which every divisor of $x$ appears, not just one associate class representative.
\end{proof}
\begin{remark} In fact, if $R$ is a $\beta$-FFR for any choice of $\beta \in \{$ associate, strong associate, very strong associate $\}$, then $R$ is pr\'esimplifiable.  This in turn forces all of the associate relations (and irreducible definitions) to coincide for non-zero, non-units.  For instance, if $R$ were not pr\'esimplifiable, then there is a non-zero $x\in R$ and a non-unit $y\in R$ such that $x=xy$.  This yields factorizations of the form $x=xy=(xy)y=(xy)yy= \cdots$ which generates a list of increasingly long factorizations which would contradict the hypothesis that $R$ were a FFR.  This same argument also shows that for $R$ to be a BFR, $R$ is also necessarily pr\'esimplifiable.  This is discussed in \cite{Valdezleon}.
\end{remark}
\indent We can use the previous results and the divisor graph for a simple proof of a result from \cite{Valdezleon} that a FFR is a BFR.

\begin{theorem} Let $R$ be a commutative ring and let $\beta \in \{$ associate, strong associate, very strong associate $\}$.  If $R$ be a $\beta$-FFR, then $R$ is a BFR.
\end{theorem}
\begin{proof} Let $R$ be a $\beta$-FFR.  Let $x \in R$ be a non-unit.  By Corollary \ref{cor: ffr}, we know that for $x\in R$, we have $\vert E(G^{\beta}_{\emptyset}(x))\vert < \infty$.  Suppose $\vert E(G^{\beta}_{\emptyset}(x))\vert=N \in \mathbb{N}$.  Then since the pseudo-clique number is the size of the edge set of the largest pseudo-clique in $ E(G^{\beta}_{\emptyset}(x))$, we certainly have $\Omega(G^{\beta}_{\emptyset}(x)) \leq N$.  This shows the pseudo-clique number of $G^{\beta}_{\emptyset}(x)$ is finite and an application of Theorem \ref{thm: bfr} implies that $R$ is a BFR as desired.
\end{proof}

\begin{theorem} \label{thm: wffr} Let $R$ be a commutative ring and let $\beta \in \{$ associate, strong associate, very strong associate $\}$.  Then we have the following.
\begin{enumerate}
\item A non-unit $x\in R$ has a finite number of divisors up to $\beta$ if and only if $V(G^{\beta}_{\emptyset}(x))$ is finite.  
\item A non-unit $x\in R$ has a finite number of divisors if and only if $V(G^{\emptyset}_{\emptyset}(x))$ is finite.
\item $R$ is a $\beta$-WFFR if and only if for all $x\in R$ not a unit, $\vert V(G^{\beta}_{\emptyset}(x)) \vert < \infty$.
\item $R$ is strong-WFFR (i.e. every non-unit has a finite number of divisors) if and only if $V(G^{\emptyset}_{\emptyset}(x))$ is finite for all non-units $x\in R$.
\end{enumerate}
\end{theorem}
\begin{proof} (1) The set of vertices of $G^{\beta}_{\emptyset}(x)$ are precisely the set of representatives, up to $\beta$, of the divisors of $x$. (2) Similarly, $V(G^{\emptyset}_{\emptyset}(x))$ is the set of all divisors of $x$. (3) This is immediate from (1) and the definition of $\beta$-WFFR.  (4) This is immediate from (2) and the definition of a strong-WFFR.
\end{proof}

\begin{theorem} \label{thm: idf} Let $R$ be a commutative ring and let $\alpha \in \{$ atomic, strongly atomic, m-atomic, very strongly atomic $\}$ and $\beta \in \{$ associate, strong associate, very strong associate $\}$.  Then we have the following.
\begin{enumerate}
\item A non-unit $x\in R$ has a finite number of $\alpha$-divisors up to $\beta$ if and only if $V(G^{\beta}_{\alpha}(x))$ is finite.  
\item A non-unit $x\in R$ has a finite number of $\alpha$-divisors if and only if $V(G^{\emptyset}_{\alpha}(x))$ is finite.
\item $R$ is a $\alpha$-$\beta$-idf ring if and only if for all $x\in R$ not a unit, $\vert V(G^{\beta}_{\alpha}(x)) \vert < \infty$.
\item $R$ is strong-$\alpha$-divisor finite ring (i.e. every non-unit has a finite number of $\alpha$-divisors) if and only if $V(G^{\emptyset}_{\alpha}(x))$ is finite for all non-units $x\in R$.
\end{enumerate}
\end{theorem}
\begin{proof} (1) The set of vertices of $G^{\beta}_{\alpha}(x)$ are precisely the set of representatives, up to $\beta$, of the $\alpha$-divisors of $x$. (2) Similarly, $V(G^{\emptyset}_{\alpha}(x))$ is the set of all $\alpha$-divisors of $x$. (3) This is immediate from (1) and the definition of $\alpha$-$\beta$-idf ring.  (4) This is immediate from (2) and the definition of a strong-$\alpha$-divisor finite ring.
\end{proof}

\indent The following theorem was proved in \cite[Proposition 6.6]{Valdezleon} by D.D. Anderson and S. Valdez-Leon.
\begin{theorem} (\cite[Proposition 6.6]{Valdezleon}) For a commutative ring $R$, the following are equivalent. 
\begin{enumerate}
\item $R$ is a FFR.
\item $R$ is a BFR and a WFFR.
\item $R$ is pr\'esimplifiable and a WFFR.
\item $R$ is a BFR and an atomic divisor finite ring.
\item $R$ is a pr\'esimplifiable and an atomic divisor finite ring.
\end{enumerate}
\end{theorem}
\indent As mentioned earlier, the conditions of FFR, BFR, and pr\'esimplifiable all have the affect of making the associate relations and irreducibles coincide.  This allows us to combine several of the previous results with \cite[Proposition 6.6]{Valdezleon} in the following theorem.
\begin{theorem} \label{thm: ffr-wffr}Let $R$ be a commutative ring and let $\alpha \in \{$ atomic, strongly atomic, m-atomic, very strongly atomic $\}$ and $\beta \in \{ $ associate, strong associate, very strong associate $\}$.  Then the following are equivalent for any (hence all) choices of $\alpha$ and $\beta$.
\begin{enumerate}
\item $R$ is a $\beta$-FFR.
\item $R$ is a BFR and a $\beta$-WFFR.
\item $R$ is pr\'esimplifiable and a $\beta$-WFFR.
\item $R$ is a BFR and a $\alpha$-$\beta$-divisor finite ring.
\item $R$ is a pr\'esimplifiable and a $\alpha$-$\beta$-divisor finite ring.
\item For all non-units, $x\in R$, we have $\sum_{a\in V(G^{\beta}_{\emptyset}(x))}\text{degl}(a) < \infty.$ 
\item For all non-units, $x\in R$, we have $\vert E(G^{\beta}_{\emptyset}(x))\vert < \infty$.
\item $R$ is a BFR and for all $x\in R$ not a unit, $\vert V(G^{\beta}_{\emptyset}(x)) \vert < \infty$.
\item $R$ is a pr\'esimplifiable and for all $x\in R$ not a unit, $\vert V(G^{\beta}_{\emptyset}(x)) \vert < \infty$.
\item $R$ is a BFR and $x\in R$ not a unit, $\vert V(G^{\beta}_{\alpha}(x)) \vert < \infty$.
\item $R$ is a pr\'esimplifiable and $x\in R$ not a unit, $\vert V(G^{\beta}_{\alpha}(x)) \vert < \infty$.
\end{enumerate}
\end{theorem}
\begin{proof} Equivalences (1)-(5) are shown to be equivalent by \cite[Proposition 6.6]{Valdezleon}.  
\\
\indent (1) $\Leftrightarrow$ (6) $\Leftrightarrow$ (7) follows from Theorem \ref{thm: ffr}.  
\\
\indent (8) (resp. (9)) is a restatement of (2) (resp. (3)) and applying the equivalence from Theorem \ref{thm: wffr}.   
\\
\indent (10) (resp. (11)) is a restatement of (4) (resp. (5)) and applying the equivalence from Theorem \ref{thm: idf}.
\end{proof}

If we are working with Noetherian rings, we can use another result, \cite[Theorem 3.9]{Valdezleon} to add even more equivalent statements to the preceding theorem.
\begin{theorem} \label{thm: valdez bfr} (\cite[Theorem 3.9]{Valdezleon}) For a Noetherian commutative ring $R$, we let $\alpha \in \{$ atomic, strongly atomic, m-atomic, very strongly atomic $\}$ and $\beta \in \{ $ associate, strong associate, very strong associate $\}$.  Then the following are equivalent for any (hence all) choices of $\alpha$ and $\beta$.
\begin{enumerate} 
\item $R$ is a BFR.
\item $R$ is pr\'esimplifiable.
\item $\cap_{i=1}^\infty (y^n)=0$ for each non-unit $y\in R$.
\item $\cap_{i=1}^\infty I^n=0$ for each proper ideal $I$ of $R$.
\end{enumerate}
\end{theorem}
The following corollary lists several more equivalent characterizations of a $\beta$-FFR for any choice of associate.
\begin{corollary} For a Noetherian ring $R$, the following conditions are equivalent.
\begin{enumerate}
\item $R$ is a $\beta$-FFR.
\item $R$ is a BFR and a $\beta$-WFFR.
\item $R$ is pr\'esimplifiable and a $\beta$-WFFR.
\item $\cap_{i=1}^\infty (y^n)=0$ for each non-unit $y\in R$ and $R$ is a $\beta$-WFFR.
\item $\cap_{i=1}^\infty I^n=0$ for each proper ideal $I$ of $R$ and $R$ is a $\beta$-WFFR.
\item $R$ is a BFR and a $\alpha$-$\beta$-divisor finite ring.
\item $R$ is a pr\'esimplifiable and a $\alpha$-$\beta$-divisor finite ring.
\item $\cap_{i=1}^\infty (y^n)=0$ for each non-unit $y\in R$ and $R$ is a $\alpha$-$\beta$-divisor finite ring.
\item $\cap_{i=1}^\infty I^n=0$ for each proper ideal $I$ of $R$ and $R$ is a $\alpha$-$\beta$-divisor finite ring.
\item For all non-units, $x\in R$, we have $\sum_{a\in V(G^{\beta}_{\emptyset}(x))}\text{degl}(a) < \infty.$ 
\item For all non-units, $x\in R$, we have $\vert E(G^{\beta}_{\emptyset}(x))\vert < \infty$.
\item $R$ is a BFR and for all $x\in R$ not a unit, $\vert V(G^{\beta}_{\emptyset}(x)) \vert < \infty$.
\item $R$ is a pr\'esimplifiable and for all $x\in R$ not a unit, $\vert V(G^{\beta}_{\emptyset}(x)) \vert < \infty$.
\item $\cap_{i=1}^\infty (y^n)=0$ for each non-unit $y\in R$ and for all $x\in R$ not a unit, $\vert V(G^{\beta}_{\emptyset}(x)) \vert < \infty$.
\item $\cap_{i=1}^\infty I^n=0$ for each proper ideal $I$ of $R$ and for all $x\in R$ not a unit, $\vert V(G^{\beta}_{\emptyset}(x)) \vert < \infty$.
\item $R$ is a BFR and $x\in R$ not a unit, $\vert V(G^{\beta}_{\alpha}(x)) \vert < \infty$.
\item $R$ is a pr\'esimplifiable and $x\in R$ not a unit, $\vert V(G^{\beta}_{\alpha}(x)) \vert < \infty$.
\item $\cap_{i=1}^\infty (y^n)=0$ for each non-unit $y\in R$ and $x\in R$ not a unit, $\vert V(G^{\beta}_{\alpha}(x)) \vert < \infty$.
\item $\cap_{i=1}^\infty I^n=0$ for each proper ideal $I$ of $R$ and for all non-units $x\in R$, $\vert V(G^{\beta}_{\alpha}(x)) \vert < \infty$.
\end{enumerate}
\end{corollary}
\begin{proof}  This theorem directly combines Theorem \ref{thm: valdez bfr} with Theorem \ref{thm: ffr-wffr}.

\end{proof}

\indent The following theorem is one of the nicest results from the work by J. Coykendall and J. Maney, in \cite{Coykendall}.  In it, the authors were studying irreducible divisor graphs in the integral domain case.
\begin{theorem}(\cite[Theorem 5.1]{Coykendall}) If $D$ is an atomic domain, then the following are equivalent.
\begin{enumerate}
\item 1. R is a UFD;
\item For each non-zero non-unit $x\in R$, $G(x)$ is a pseudo-clique;
\item For each non-zero non-unit $x\in R$, $\overline{G}(x)$ is a clique;
\item For each non-zero non-unit $x\in R$, $G(x)$ is connected.
\end{enumerate}
\end{theorem}

\indent Again, a $\alpha$-$\beta$-UFR is certainly a $\beta$-FFR which is a BFR and hence pr\'esimplifiable.  Again all of the associate relations coincide and irreducible, strongly irreducible, m-irreducible and very strongly irreducible coincide for any choice of a $\alpha$-$\beta$-UFR.  This is discussed by D.D. Anderson and S. Valdez-Leon preceding Definition 4.3 in \cite{Valdezleon}.  This leads to the following result.
\begin{theorem}\label{thm: UFR AV} (\cite[Theorem 4.4]{Valdezleon})  Let $R$ be a commutative ring and for any (and all) choice of $\alpha \in \{$ atomic, strongly atomic, m-atomic, very strongly atomic $\}$ and $\beta \in \{$ associate, strong associate, very strong associate $\}$, then the following are equivalent.
\begin{enumerate}
\item $R$ is a $\alpha$-$\beta$-UFR.
\item $R$ is either (a) a UFD (b) an SPIR or (c) a quasi-local ring with $M^2=0$ where $M$ is the unique maximal ideal of $R$.  
\item $R$ is a UFR in the sense of A. Bouvier in \cite{bouvier74b}. 
\item $R$ is a UFR in the sense of S. Galovich in \cite{Galovich}.
\end{enumerate}
\end{theorem}

\begin{theorem} \label{thm: UFR} Let $R$ be a commutative ring and let $\alpha \in \{$ atomic, strongly atomic, m-atomic, very strongly atomic $\}$ and $\beta \in \{$ associate, strong associate, very strong associate $\}$.  If $R$ satisfies any of the following equivalent conditions:
\begin{enumerate}
\item $R$ is a $\alpha$-$\beta$-UFR,
\item $R$ is either (a) a UFD (b) an SPIR or (c) a quasi-local ring with $M^2=0$ where $M$ is the unique maximal ideal of $R$,
\item $R$ is a UFR in the sense of A. Bouvier in \cite{bouvier74b}, or
\item $R$ is a UFR in the sense of S. Galovich in \cite{Galovich}, 
\end{enumerate}
then for any non-unit $x\in R$, $\overline{G}_\alpha^\beta (x) \cong K_{N(x)}$ for some $N(x) \in \N$, where $K_n$ is the complete graph on $n$ vertices.  Moreover, $G_\alpha^\beta (x)$ is a pseudo-clique.  
\end{theorem}
\begin{proof} By Theorem \ref{thm: UFR AV}, (1)-(4) are equivalent, so we let $R$ be a $\alpha$-$\beta$-UFR.  Let $x\in R$ be a non-unit.  Let $x=a_1 \cdots a_n$ be the unique $\alpha$-factorization up to $\beta$.  We suppose $a_1, \ldots a_s$ with $s \leq n$ are distinct up to $\beta$. We may now group like factors up to $\beta$ and rewrite the $\alpha$-factorization as $x=a_1^{e_1}a_2^{e_2} \cdots a_s^{e_s}$ with $e_i \geq 1$ and $e_1 + e_2 + \cdots e_s=n$.  Since this is the only $\alpha$-factorization of $x$ up to $\beta$, we have $V(G_\alpha^\beta(x))=\{a_1, \ldots a_s \}$. We see $a_ia_j \in E(G_\alpha^\beta(x))$ for all $i \neq j$ and there are $e_i -1$ loops on vertex $a_i$.  This proves that $G_\alpha^\beta (x)$ is a pseudo-clique.  We set $N(x)=s$ and see that indeed $\overline{G}_\alpha^\beta (x) \cong K_{s}$ as desired.
\end{proof}
\indent Unfortunately, the full analogues of \cite[Theorem 5.1]{Coykendall} will not hold with zero-divisors as the next example demonstrates.
\begin{example} Let $R=\Z/2\Z \times \Z/2\Z$ and let $\alpha \in \{ \ \emptyset$, atomic, strongly atomic, m-atomic $\}$ and let $\beta \in \{\ \emptyset$, associate, strongly associate $\}$. \end{example}\ 
\\
\indent We note that for all $i \geq 1$ we have
$$(0,0)=(1,0)^i(0,1)=(1,0)(0,1)^i$$
as the only valid non-trivial $\alpha$ factorizations of $(0,0)$.  Moreover, the only factorizations of $(0,1)$ are of the form $(0,1)=(0,1)^i$.  Certainly $(0,1) \sim (0,1)$ and $(0,1) \approx (0,1)$ showing $(0,1)$ is atomic, strongly atomic, and m-atomic.  Similarly, the only factorizations of $(1,0)$ are of the form $(1,0)=(1,0)^i$.  Again, $(1,0) \sim (1,0)$ and $(1,0) \approx (1,0)$ showing $(1,0)$ is atomic,strongly atomic, and m-atomic.  
\\
\indent We note that $(0,1)$ and $(1,0)$ are not very strongly atomic since they are non-trivial idempotents since $(1,0)=(1,0)(1,0)$ and $(0,1)=(0,1)(0,1)$ are non-trivial factorizations of $(1,0)$ and $(0,1)$ respectively.  
\\
\indent For $\alpha \in \{ \ \emptyset$, atomic, strongly atomic, m-atomic $\}$ and $\beta \in \{\ \emptyset$, associate, strongly associate $\}$, we have the following divisor graphs.
\begin{figure}[H]
	 \centering
		 \includegraphics[scale=.7]{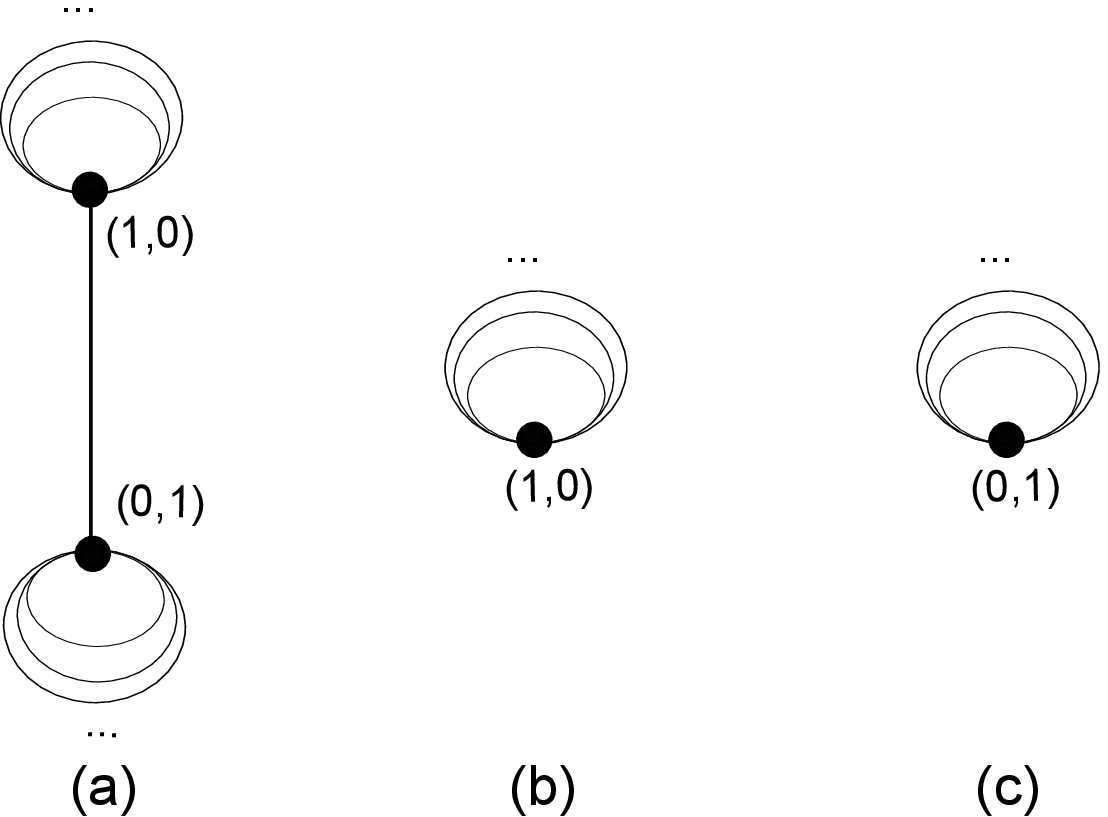}
	 \caption{(a) $G_\alpha^\beta \left( (0,0)\right)$\ \ \ (b) $G_\alpha^\beta \left( (1,0)\right)$\ \ \  (c) $G_\alpha^\beta \left( (0,1)\right)$}
	 \label{fig:Irreducible divisor graph 00}
 \end{figure}

\indent This shows that while the $\alpha$-$\beta$-divisor graphs of $(0,0)$, $(1,0)$ and $(0,1)$ are complete, connected and have a finite number of vertices (albeit with an infinite number of loops on each vertex), $R$ is neither a $\alpha$-$\beta$-UFR, $\alpha$-HFR, $\beta$-FFR nor even a BFR.  
\\
\indent This example also demonstrates that the converse of Theorem \ref{thm: accp} will not hold.  A finite ring certainly satisfies ACCP, on the other hand, all vertices have infinite degree when you include loops.

\section*{Acknowledgment}
The author would like to acknowledge that some of this research was conducted under the supervision of Professor Daniel D. Anderson while serving as a Presidential Graduate Research Fellow at The University of Iowa.

\end{document}